\documentclass[11pt]{article}
\usepackage[T1]{fontenc}
\usepackage[utf8]{inputenc}
\usepackage{amsmath,amssymb,euscript}
\usepackage{bbm}
\usepackage{amsthm}
\usepackage[english]{babel}
\theoremstyle{plain}
\newtheorem{theorem}{Theorem}
\newtheorem{lemme}{Lemma}
\newtheorem{prop}{Proposition}

\usepackage{amsfonts}
\usepackage{amssymb}
\usepackage{makeidx}
\usepackage{graphicx}
\usepackage{stmaryrd}
\usepackage{caption,subcaption}
\usepackage{color}

\newcommand{\bbD}{{\ensuremath{\mathbb D}} }
\newcommand{\bbE}{{\ensuremath{\mathbb E}} }

\newcommand{\bbN}{{\ensuremath{\mathbb N}} }

\newcommand{\bbP}{{\ensuremath{\mathbb P}} }

\newcommand{\bbR}{{\ensuremath{\mathbb R}} }

\newcommand{\bbT}{{\ensuremath{\mathbb T}} }

\newcommand{\cB}{\ensuremath{\mathcal B}}

\newcommand{\cE}{\ensuremath{\mathcal E}}

\newcommand{\cG}{\ensuremath{\mathcal G}}

\newcommand{\cL}{\ensuremath{\mathcal L}}

\newcommand{\cN}{\ensuremath{\mathcal N}}

\newcommand{\cP}{\ensuremath{\mathcal P}}

\newcommand{\cU}{\ensuremath{\mathcal U}}

\newcommand{\1}[1]{\bf{1}_{\{#1\}}}							
\newcommand{\va}[1]{\vert #1 \vert}							
\newcommand{\p}[1]{\parallel #1 \parallel}					

\newcommand{\E}[1]{\bbE \big[#1 \big]}						
\newcommand{\tot}[1]{#1 \times \{0,1\}}

\newcommand{\Rp}[1]{\bbR_+^{#1}}                            

\def\rk{{r_{K}}}                                          
\def\rkm{{r_{K,M}}}

\newcommand{\cm}[1]{c_{#1,M}}

\newcommand{\K}[1]{#1^{K}}                                

\newcommand{\dint}[2]{\int_{0}^{#1} \int_{#2}}				
\def\esp{{\hspace{1cm}}}									
\newcommand{\DT}[1]{\mathbb{D}([0,T], #1)}					

\def\me{\medskip \noindent}
\def\bi{\bigskip \noindent}
\def\sm{\smallskip \noindent}
\def\be{\begin{eqnarray}}
\def\ee{\end{eqnarray}}
\def\ben{\begin{eqnarray*}}
	\def\een{\end{eqnarray*}}

\title{A PDMP to model the stochastic influence of quiescence dynamics in blood cancers}
\author{C\'eline Bonnet\thanks{ENSL, UMPA, CNRS UMR 5669, 69364 Lyon, France. E-mail: \texttt{celine.bonnet@ens-lyon.fr}}}

\date{}
\begin{document}
	\maketitle
	\begin{abstract}
		
		 In this article, we will see a new approach to study the impact of a small microscopic population of cancer cells on a macroscopic population of healthy cells, with an example inspired by pathological hematopoiesis. Hematopoiesis is the biological phenomenon of blood cells production by differentiation of cells called hematopoietic stem cells (HSCs). We will study the dynamics of a stochastic $4$-dimensional process describing the evolution over time of the number of healthy and cancer stem cells and the number of healthy and mutant red blood cells. The model takes into account the amplification between stem cells and red blood cells as well as the regulation of this amplification as a function of the number of red blood cells (healthy and mutant). A single cancer HSC is considered while other populations are in large numbers. We assume that the unique cancer HSC randomly switches between an active and a quiescent state. We show the convergence in law of this process towards a piecewise deterministic Markov process (PDMP), when the population size goes to infinity. We then study the long time behaviour of this limit process. We show the existence and uniqueness of an absolutely continuous invariant probability measure with respect to the Lebesgue's measure for the limit PDMP, previously gathered. We describe the support of the invariant probability and show that the process converges in total variation towards it, using theory develop in \cite{benaim2015qualitative} and \cite{benaim2018user}. We finally identify the invariant probability using its infinitesimal generator. Thanks to this probabilistic approach, we obtain a stationary system of partial differential equation describing the impact of cancer HSC quiescent phases and regulation on the cell density of the hematopoietic system studied.
		
	\end{abstract}
	

	\emph{\textbf{MSC classes:} 60F05, 60J28, 92C32.}
	
	\section{Introduction}

	We will see a method to study the interaction between macroscopic populations and a small population with a stochastic dynamics, inspired by pathological hematopoiesis. Hematopoiesis refers to the production of blood cells by differentiation of hematopoietic stem cells (HSCs). The HSCs produce a large number of blood cells every day, in particular red blood cells. Myeloproliferative Neoplams are blood cancers in which some cancer HSCs lead to an overproduction of red blood cells and a perturbation of the whole system through regulations.
	These symptoms seem to conflict with the dynamics of quiescence of cancer HSCs, which can become inactives and no longer produce cancerous blood cells for a random time. It is important to be able to describe the influence of cancer HSC quiescence since it plays a role in the resistance of these cells to chemotherapy. In this article, we will give a probabilistic approach to obtain an equation describing the state of the hematopoietic system depending on the quiescence dynamics of one cancer HSC and regulation. 
	
	\me More precisely, we will obtain a stationary system of partial differential equations (PDEs) describing cell density of $3$ macroscopic populations depending on regulation and on the quiescence dynamics of one cancer HSC. The three populations described are the healthy HSCs, the healthy red blood cells produced by them and the cancer red blood cells produced by the only cancer HSC when it is active.
	
	\bigskip Numerous mathematical modeling have been developed to understand the dynamics of cancer HSCs (\cite{fasano2017blood,michor2008mathematical}). These models are stochastic (\cite{catlin2005kinetics,dingli2011stochastic,dingli2007symmetric,komarova2007effect}) or deterministic (\cite{ANDERSEN2020,besse2018stability,dingli2006successful,stiehl2012mathematical}) and do not take into account both deterministic and stochastic dynamics as we will see. Our model assumes the existence of a single cancer HSC while the resident (healthy) cells are in large number. We will see that these different size scales imply a difference in the time scales in which the cell dynamics occur and lead us to take into account the both dynamics.
	
	Other authors combine stochastic and deterministic methods to model the dynamics of one cell type (\cite{horn2008mathematical,roeder2006dynamic,haeno2009progenitor,glauche2012therapy}). Their objective is to approach a certain biological reality. Bangsgaard et al \cite{Bangsgaard2020} use the same model as \cite{ANDERSEN2020} in which they add the possibility for stem cells to acquire mutations randomly over time. All the cell types are considered in large populations with a deterministic dynamics. For a review of mathematical models on cancer HSC dynamics, we refer to the recent article \cite{pedersen2023}.
	
	\medskip Contrary to these models, we will highlight the difference in size scales that exists between cell populations using a scale parameter $K$. This parameter represents the number of healthy HSCs, assumed to be constant over time, and is used to quantify the high production of red blood cells by differentiation of healthy and cancer HSC. We will study the limit in law when $K$ tends to infinity of a Markov process describing the dynamics of the $4$ populations mentioned (healthy and mutant red blood cells, and healthy and mutant HSC). We will obtain a piecewise deterministic Markovian process (PDMP) which describe describe the dynamics of each cell populations depending on the random switch of the unique cancer HSC. Such a convergence, for which the limit admits discrete-valued component, has already been studied by Crudu et al (cf \cite{crudu}, Th.3.1). We have chosen to demonstrate this result by a more direct proof (cf Th.\ref{cvK}).

\me The study of the long time behavior of this process allows us to 
show existence and uniqueness of a invariant probability measure for the limit PDMP. We will deduce from this result, existence and uniqueness of a weak solution of a stationary system of partial differential equations which describe the cell density of the $3$ macroscopic populations, previously introduced, depending on cancer HSC quiescence dynamics and regulation.

	\medskip PDMPs are stochastic processes involving deterministic motion punctuated by random jumps. To a detailed description of such class of models, we refer to Davis's work \cite{davis1984piecewise,davis1993markov} (see also \cite{jacobsen2006point}). These processes are often used in mathematical modeling. They occur in biology (\cite{cloez2017probabilistic}), epidemiology (\cite{li2017threshold}), ecology (\cite{benaim2016lotka}), bacterial movement (\cite{fontbona2012quantitative}) or gene expression (\cite{mackey2013dynamic,yvinec2014adiabatic}). Each of these models admits different particularities and difficulties concerning the study of their behaviour in long time (\cite{malrieu2015some}).
	
	\bigskip We structured our mathematical results in four different sections.
	
	\me In Section \ref{chap3:sect:cv}, we present the model using Poisson point measures. By a moment control of the previously rescaled process, we show that it admits a decomposition in semi-martingales and that it converges in law, when $K$ tends to infinity, toward a PDMP (cf Th.\ref{cvK}).
	
	 \noindent We then study the long time behaviour of the limiting PDMP in Section \ref{chap3:sect:pdmp}. Following the main steps of Benaïm et al article \cite{benaim2015qualitative}, we show a set of intermediate results that permit to demonstrate the existence and uniqueness of an invariant probability measure of the limiting PDMP (cf Th.\ref{th-mes-inv}), using \cite{benaim2018user} Corollary 2.7. This measure is absolutely continuous with respect to Lebesgue's measure. The PDMP converges to it exponentially fast in total variation (cf Th.\ref{th-mes-inv}). The proof of Theorem \ref{th-mes-inv} is mainly based on the check of a weak bracket (or Hörmander) condition by the vector fields associated to the PDMP.
	
	\noindent Finally in Section \ref{chap3:sect:pers}, we identify the associated density as a couple of integrable and positive functions solution of a stationary system of partial differential equations (Th. \ref{ident}).

	\bi \textit{Notation.}
	
	\me We introduce the set $E= [0,1] \times \Rp{2}$ with norm $ \p{x}^2 = \displaystyle \sum_i x_i^2$.
	
	\sm Then let us introduce $\cE = E \times \{0,1\}$ and $C^1_b(\cE)$ the following functions set
	\begin{equation}
	\label{C1b}
	C^1_b(\cE)= \{f : \cE \to \bbR \text{ borned with a $C^1$ restriction to } E\}.
	\end{equation}

	\section{The model and its macroscopic approximation}
	\label{chap3:sect:cv}
	
	\me In this section, we will present our model and its dynamics assumptions first. We will study a Markovian process $(N^K,I^K)=(N^K_1,N^K_2,N^K_3,I^K)$ which respectively describe the number of active healthy HSCs, of healthy red blood cells, of mutant red blood cells and the state of the unique cancer HSC over time. 
	
	\me As explained in Introduction, we assume that the total number of healthy HSCs (quiescent and active) is constant over time and equal to a scale parameter $K \in \bbN^*$. This parameter is assumed to be large. Moreover we assume that the stochastic process $I^K$ is equal to $0$ when the unique cancer HSC is quiescent and to $1$ when it is active. Hence for any $t \in \bbR^+$, $$N_1^K(t) \in [0,K] \quad \text{ and } \quad I^K(t) \in \{0,1\}.$$
	Let us note that the number of quiescent healthy HSCs at time $t$ is equal to $K-N^K_1(t)$. 
	
	\me We assume that healthy (respectively cancer) HSC switch from active state to quiescent state at a constant rate $a>0$ (respectively $a_M>0$). They switch from quiescent state to active state at rate $q^K$ (respectively $q^K_M$). The function rates $q^K$ and $q^K_M$ are increasing depending on the number of healthy red blood cells and the number of mutant red blood cells. An active HSC generates a red blood cell by asymmetric division at rate $\tau$ for healthy cells and at rate $\tau_M$ for mutant cells. These division rates are regulated by the numbers of healthy and mutant red blood cells. They depend on $K$ as follows, $$\tau =K^{\alpha}\, r^K, \quad \tau_M =K^{\beta}\, r_M^K $$ with $\alpha>0$ and $\beta>0$. The functions $r^K$ and $r^K_M$ are decreasing and bounded. They model, respectively, the regulation of the production of healthy and mutant red blood cells as a function of the number of both healthy and mutant red blood cells in the system. The explicit forms of the regulation function rates $q^K$, $q_M^K$, $r^K$ and $r_M^K$ are given in \eqref{deftxK}. The powers $\alpha$ and $\beta$ model the large number of red blood cells produced by the HSCs, respectively for healthy and mutant cells.
	
	\me The red blood cells have a constant individual death rate, $d>0$ for the healthy cells and $d_M>0$ for the mutant cells. 
	
	\me Now, we are looking for an appropriate size scale in order to study the stochastic process $(N^K,I^K)$. Indeed, we have assumed that the number of healthy HSCs is constant over time and equal to $K$. Moreover, the amplification between the HSC and red blood cell compartments is modeled by a multiplicative factor $K^{\alpha}$ for healthy and $K^{\beta}$ for mutant cells, respectively. Hence, we will see in Lemma \ref{control} that these factors induced the order of magnitude of each component size $$N^K_1 \sim K, \quad N^K_2\sim K^{1+\alpha}, \quad N^K_3\sim K^{\beta}.$$
	This first result highlights an appropriate size scaling allowing to study the limits of the processes $N^K$ and $I^K$ when $K$ tends to infinity (cf Th.\ref{cvK}).\\
	
	Let us first describe more precisely the different regulations of the system. We assume that the function rates $q^K$, $q^K_M$, $r^K$ and $r^K_M$ are given, for any $(n_2,n_3) \in \bbN^{2}$, by
	
	\begin{align}
	\label{deftxK}
	\left\lbrace
	\begin{array}{l}
	\displaystyle q^K(n_2,n_3)=q(\frac{n_2}{K^{1+\alpha}},\frac{n_3}{K^{\beta}}) \\ \\
	\displaystyle q_M^K(n_2,n_3)=q_M(\frac{n_2}{K^{1+\alpha}},\frac{n_3}{K^{\beta}})\\ \\
	\displaystyle r^K(n_2,n_3)=r(\frac{n_2}{K^{1+\alpha}},\frac{n_3}{K^{\beta}}) \\ \\
	\displaystyle r_M^K(n_2,n_3)=r_M(\frac{n_2}{K^{1+\alpha}},\frac{n_3}{K^{\beta}})\\ 
	\end{array}
	\right.
	\end{align}
	where for any $(x_2,x_3) \in \Rp{2}$, 
	\begin{align}
	\label{deftx}
	\left\lbrace
	\begin{array}{l}
	q(x_2,x_3)=q_1+q_2\,x_2+q_3\,x_3, \\ \\
	q_M(x_2,x_3)=q_{1,M}+q_{2,M}\,x_2+q_{3,M}\,x_3 \\ \\
	\displaystyle r(x_2,x_3)= \frac{c_1}{1+c_2\,x_2+c_3\,x_3},\\ \\
	\displaystyle r_M(x_2,x_3)=\frac{\cm{1}}{1+\cm{2}\,x_2+\cm{3}\,x_3}
	\end{array}
	\right.  .
	\end{align}
	
	\me  We assume that
	\begin{itemize}
		\item[$(H1)$] The parameters $a$, $q_1$, $q_{1,M}$, $c_1$, $c_{1,M}$, $d$ and $d_M$ are strictly positive;
		\item[$(H2)$] The parameters $q_2$, $q_3$, $c_2$, $c_3$, $q_{2,M}$, $q_{3,M}$, $c_{2,M}$, $c_{3,M}$ are positive.
	\end{itemize}
	

	
	\me For any $K$, the process $N^K$ defines a Markov jump process whose dynamics is described by the following stochastic system. 
	
	\me Let $(\cN^j_i)_{\underset{j \in \{+,-\}}{i \in \{1,2,3,4\}}}$ be independent Poisson point measures on $(\Rp{} \times \Rp{}, \; \cB(\Rp{}) \otimes \cB(\Rp{}))$ with intensity $dsdu$.
	
	\me Let $(\mathcal{F}_t)_{t \geq 0}$ be the associated filtration, $$\mathcal{F}_t = \sigma(\cN^j_i([0,s),A); \, i \in \{1,2,3,4\},  j \in \{+,-\} \,s \leq t, \, A \in \cB(\Rp{}) ).$$
	
	Then for any $t\geq0$, we define
	
	\begin{align}
	\label{nk}
	N^K_1(t) =\, N^K_1(0) &+ \dint{t}{\Rp{}} \1{\,u \,\leq \,a \, \big(\,K - N^K_1(s^-)\,\big) \,} \cN^+_1(ds,du) \nonumber\\ &- \dint{t}{\Rp{}} \1{\,u \,\leq \, q^K(N_2^K(s^-),N_3^K(s^-)) \, N^K_1(s^-)\,} \cN^-_1(ds,du)  \nonumber\\
	N^K_2(t) =\, N^K_2(0) &+ \dint{t}{\Rp{}} \1{u \leq K^{\alpha} \, \rk(N_2^K(s^-),N_3^K(s^-))N^K_1(s^-)} \cN^+_2(ds,du) \\&- \dint{t}{\Rp{}} \1{u \leq d N^K_2(s^-)} \cN^-_2(ds,du) \nonumber	\\ \nonumber
	N^K_3(t) =\, N^K_3(0) &+ \dint{t}{\Rp{}} \1{\,u \,\leq\, K^{\beta} \, \rkm(N_2^K(s^-),N_3^K(s^-))\,I^K(s^-) \,} \cN_{3}^+(ds,du) \\ &-\dint{t}{\Rp{}} \1{\,u\, \leq \,d_M\, N^K_3(s^-)\,} \cN_{3}^-(ds,du) \nonumber\\
	I^K(t) =\, I^K(0) &+ \dint{t}{\Rp{}} \1{\,u \,\leq\, a_M \,\big(\,1 - I^K(s^-)\,\big)} \cN_{4}^+(ds,du)   \nonumber\\ & - \dint{t}{\Rp{}} \1{\,u \,\leq\, q^K_M(N_2^K(s^-),N_3^K(s^-)) \,I^K(s^-)\,} \cN_{4}^-(ds,du) \nonumber.\\\nonumber
	\end{align}
	where the functions $q^K$, $q_M^K$, $r^K$ and $r^K_M$ are introduced in \eqref{deftxK}-\eqref{deftx}. 
	
\me Let us also define the stochastic process $X^K$ by  \begin{equation}\label{defXK}\K{X}=(\frac{N_1}{K}, \frac{N^K_2}{K^{1+\alpha}},\frac{N_3^K}{K^{\beta}}).\end{equation}
	Let us firstly state a uniform control for 2 order moment of $(X^K)_K$.
	\begin{lemme}
	\label{control}
	We assume that the $\cE$-valued random vector $(\K{X}(0),I^K(0))$ satisfies \begin{equation}\label{hyplem}\displaystyle \sup_K \E{\p{X^K(0)}^2}< \infty.\end{equation} Then for any $T>$0, 
	\begin{equation*}
		\displaystyle \sup_K \E{ \sup_{t \leq T} \p{X^K(t)}^2 } < \infty.
	\end{equation*}
\end{lemme}
\begin{proof}
Using Itô's formula, \eqref{deftxK}-\eqref{deftx}, a localization argument and Gronwall's lemma (see for example \cite{CoursMeleardVincent}), we easily obtain, for any $T>0$ and $K \in \bbN^*$, $$\E{\displaystyle \sup_{t \in [0,T]} (X^K_2(t))^2} \leq (\E{(X^K_2(0))^2} + 3\,T)\, e^{3\,c_1\,T} <\infty $$ and $$\E{\displaystyle \sup_{t \in [0,T]} (X^K_3(t))^2} \leq (\E{(X^K_3(0))^2} + 3\,T)\, e^{3\,c_{1,M}\,T} <\infty.$$

Then the result follows using \eqref{hyplem}. 

\end{proof}

\me Finally we can state the main result of this section describing the asymptotic first-order behavior of the process $(X^K,I^K)$ over a finite time interval.

\begin{theorem}
\label{cvK}
Let $T>0$ and $\K{X}$ be the stochastic process valued in $\DT{E}$ defined in \eqref{defXK}
We assume that the sequence of random vectors $(\K{X}(0),I^K(0))_K$ converges in law to $(x_0,i_0) \in \cE$ and satisfies \begin{equation}\label{hyp0}\displaystyle \sup_K \E{\p{X^K(0)}^2}< \infty.\end{equation}
Then the sequence $((X^{K}(t),I^K(t)),\; t\in [0,T])_{K}$ converges in law in $\DT{\cE}$, when $K$ tends to infinity, towards the stochastic process $(X,I)$ with initial condition $(x_0,i_0) \in \cE$ and infinitesimal generator $\cL$ defined for $f \in C^1_b(\cE)$ by 

$\forall x \in E$, $i \in \{0,1\}$,
\begin{equation}\label{L}\cL f(x,i) = \displaystyle \sum_{j=1}^3 \big( \frac{\partial f}{\partial x_j} (x,i) \, g_j(x,i) \big) \;+\, a_M\,(1-i) ( f(x,i+1)-f(x,i)) + q_M(x_2,x_3) i (f(x,i-1) - f(x,i)).\end{equation}
The functions $g$ is defined as \begin{equation}\label{g}
g(x,i) =\big( a -(a + q(x_2,x_3)) x_1\, , r(x_2,x_3)x_1 - d x_2\, , r_M(x_2,x_3) i - d_M x_3 \big)^T.
\end{equation} The functions $q$, $q_M$, $r$ and $r_M$ are defined in \eqref{deftx} and the set $C^1_b(\cE)$ in \eqref{C1b}.
\end{theorem}

The limiting process $(X,I)$ is called a Piecewise Deterministic Markov Process (PDMP). The process $I$ randomly switches between $0$ and $1$, whereas the process $X$ has almost surely continuous trajectories deterministically defined between two switches of $I$, as the unique solution of the equation $\frac{dX(t)}{dt} = g(X(t),I(t))$. We will see in Section \ref{chap3:sect:pers} how can be constructed such a process.

\me Now let us prove Theorem \ref{cvK}.
\begin{proof} 
	We deduce from \eqref{hyp0}, Lemma \ref{control} and \eqref{nk} the following decomposition in semi-martingales of the process $(X^K,I^K)$.
	
	\sm $\forall t\geq 0$,
	\begin{align}
	\label{semiM}
	X^K_1(t) =\, X^K_1(0) &+ \int_{0}^{t} \Big(a \, (\,1 - X^K_1(s)\,) - q(X_2^K(s),X_3^K(s))X^K_1(s) \Big)\, ds + M^K_1(t) \nonumber\\
	X^K_2(t) =\, X^K_2(0) &+ \int_{0}^{t} \Big(r(X_2^K(s),X_3^K(s))X^K_1(s)- d X^K_2(s) \Big)\, ds + M^K_2(t)	\\
	X^K_3(t) =\, X^K_3(0) &+ \int_{0}^{t} \Big( r_M(X_2^K(s),X_3^K(s))I^K(s) -  d_M X^K_3(s) \Big)\,ds + M^K_3(t)\nonumber\\
	I^K(t) =\, N^K_3(0) &+ \int_{0}^{t} \Big(a_M \,(\,1 - I^K(s)\,) -  q_M(X_2^K(s),X_3^K(s)) I^K(s) \Big)\,ds + \widehat{M^K}(t),\nonumber
	\end{align}
	where $M_i^K$ and $\widehat{M^K}$ are square integrable martingales with the following quadratic variations
	\begin{align}
	\label{crocM}
	\langle M^K_{1} \rangle_{t} &=\, K^{-1}\, \int_{0}^{t} \Big(a \, (\,1 - X^K_1(s)\,) + q(X_2^K(s),X_3^K(s)) \,X^K_1(s) \Big)\, ds  \nonumber\\
	\langle M^K_{2} \rangle_{t} &=\, K^{-(1+\alpha)}\, \int_{0}^{t} \Big(r(X_2^K(s),X_3^K(s))X^K_1(s)+ d X^K_2(s) \Big)\, ds	\\
	\langle M^K_{3} \rangle_{t} &=\, K^{-\beta} \,\int_{0}^{t} \Big( r_M(X_2^K(s),X_3^K(s))I^K(s) +  d_M X^K_3(s) \Big)\,ds \nonumber\\
	\langle \widehat{M^K} \rangle_{t} &=\,  \int_{0}^{t} \Big(a_M \,(\,1 - I^K(s)\,) +  q_M(X_2^K(s),X_3^K(s)) I^K(s) \Big)\,ds\nonumber\\
	\langle M^K_{i},M^K_{j} \rangle_{t} &=\langle M^K_{i},\widehat{M^K} \rangle_{t}=\,0 \quad \text{ pour } i \neq j.\nonumber
	\end{align}
	
	Then we can easily check the Aldous and Rebolledo tightness criteria (see \cite{JoffeMetivier} and \cite{CoursMeleardVincent}). We deduce the uniform tightness of the sequence of law of $(X^K,I^K)_K$ in $\cP(\DT{\Rp{4}})$, the set of probabilities on $\DT{\Rp{4}}$.
	According to Prohorov's Theorem (see \cite{billingsley2013convergence}), there exists a limiting probability measure $\mu$ toward which a sub-sequence of $(X^K,I^K)_K$ converges. In the following, $(X^K,I^K)_K$ will denote this sub-sequence by simplicity. Let us now identify this measure $\mu$.
	
	\me We know that
	$$\displaystyle \sup_{t\in [0,T]} \p{\Delta X^K(t)} \leq \sqrt{K^{-2}+K^{-2(1+\alpha)}+ K^{-2\beta}} \leq \sqrt{3} \; K^{-(1\wedge\beta)}.$$
	Thus by continuity of the function $(x,i) \in\bbD([0,T],\bbR^4) \to \displaystyle \sup_{t\in [0,T]} \p{\Delta x(t)} \in \Rp{}$, we deduce that the limiting measure $\mu$ only loads the set of $\Rp{4}$-valued processes with continuous three first components.

	\me Further, from Doob's inequality, we know that for any $T>0$, 
	\begin{align*}\E{\sup_{t \leq T} \va{M^K_i(t)}^2}\leq 4 \,\E{<M^K_i>_T} .
	\end{align*}
	Then, we obtain by (\ref{crocM}) and Lemma \ref{control}, that
	$\displaystyle  \lim_{K \to \infty} \E{ \sup_{t \in [0,T]} \va{ M^K_i(t)}^2} =0$. Using Markov's inequality, we deduce that the three sequences of martingales $((M^K_i)_K,\quad i\in \{1,2,3\})$  converge in probability and for the uniform norm, to $0$.
	
	\me The infinitesimal generator associated with the process $(X^K,I^K)$  is given, for any $f \in C^1_b(\cE)$ and $(x,i)\in\cE$, by
	\begin{align*}  \cL^K(f)(x,i) = &\big( f(x+e_1\,K^{-1},i) - f(x,i) \big) \, a \, (1-x_1) \, K + \big( f(x-e_1\,K^{-1},i) - f(x,i) \big) \, q(x_2,x_3) \, x_1 \, K\\ & + \big( f(x+e_2\,K^{-(1+\alpha)},i) - f(x,i) \big) \, r(x_2,x_3) \, x_1 \, K^{1+\alpha}\\ & + \big(f(x+e_3\,K^{-\beta},i) - f(x,i) \big) \, r_M(x_2,x_3) \, i \, K^{\beta}\\& + \big( f(x,i+1) - f(x,i) \big) \, a_M \, (1-i) + \big( f(x,i-1) - f(x,i) \big) \, q_M(x_2,x_3) \, i. \end{align*}
	Then, by a Taylor expansion, we obtain \be\label{Cvgene}\forall f \in C^1_{b},\quad \lim_{K \to \infty} \sup_{(x,i) \in \cE} \va{ \cL^K(f)(x,i) - \cL(f)(x,i)}=0\ee where $\cL$ has been defined in \eqref{L}.
	
	\me Let us introduce the function $\xi^{K,f}_t$ on $\DT{\cE}$, for $f \in C^1_b(\cE)$, by   $$ \xi^{K,f}_t(x,i) = f(x_t,i_t) - f(x_0,i_0)- \int_0^t \cL^K(f)(x_s,i_s)\,ds.$$
	We deduce from Dynkin's formula, \eqref{hyp0} and Lemma \ref{control}, that $(\xi_t^{K,f}(X^K,N^K_3))_K$ defines a sequence of uniformly integrable martingales.
	 
	\me Let us define $(X,I)$ as the canonical process under $\mu$. Then, by studying the limits when K tends to infinity of $(\xi_t^{K,f}(X^K,N^K_3))_K$, we deduce from \eqref{Cvgene} that the process $(X,I)$ is a solution of a coupling between the following Cauchy and martingale problems
	
	\begin{equation}
	\label{lim}
	\left\lbrace   \begin{array}{l}
		\displaystyle \frac{dX(t)}{dt}=g(X(t),I(t))\\ 
		\\f(I_t) - f(I_0) - \int_0^t  \big( f(I_s+1) - f(I_s) \big) \, a_M \, (1-I_s) \,ds \\+ \int_0^t \big( f(I_s-1) - f(I_s) \big) \, q_M(X_2(s),X_3(s)) \, I_s\,ds \quad \text{ is a martingale}\\
	\end{array}
	\right. .
	\end{equation}
	
	\me Furthermore, we deduce from the martingale problem the existence of a pure jump martingale $M$ such that
	\begin{equation}
	\label{Ma}
	\forall t\in [0,T], \esp I(t) = I(0) + \int_{0}^{t } a_M - \Big(a_M+ q_M(X_2(s),X_3(s))\Big) I(s) ds + M(t)\in \{0,1\}.
	\end{equation} 
	We can then apply Itô's formula to $I^2(t)$. Since for all $t$, $ I(t) \in \{0,1\}$, $I^2=I$ and we obtain, by unicity of the Doob-Meyer semi-martingale decomposition, that $M$ is an integrable martingale with quadratic variation given by 
	$$<M>_t=\int_0^t a_M (1-I(s)) +q_M(X_2(s),X_3(s)) I(s) \,ds.$$
	
	Such a square integrable martingale $M$ is unique (cf \cite{lepelmar} Th.22 p.66). We deduce from this latter the existence of a Poisson point measure $\textbf{N}$ on $(\Rp{} \times \Rp{}, \; \cB(\Rp{}) \otimes \cB(\Rp{}))$ with intensity $dsdu$ such that $$I(t) =I(0) + \int_0^t \int_{\Rp{}} \1{u \leq a_M (1-I(s^-))} - \1{ a_M (1-I(s^-)) < u \leq a_M + (q_M(X_2(s),X_3(s))-a_M) I(s^-)} N(du,ds).$$
	Finally, according to the Cauchy-Lipschitz Theorem, the solution of \eqref{lim} pathwise is unique.
	
	\me The pathwise uniqueness of $(X,I)$ implies the uniqueness of the limiting law $\mu$. We then deduce the convergence in law of the process $(X^K,I^K)$ in $\DT{\cE}$ to the PDMP with infinitesimal generator $\cL$ defined by \eqref{L}.

\end{proof}

\section{Long time behaviour of the limiting process.}
\label{chap3:sect:pdmp}

\me The aim of this section is to prove the following result.

\begin{theorem}
	\label{th-mes-inv} We assume $(H1)$ and $(H2)$ (cf Section \ref{chap3:sect:cv}) and \begin{equation}
	\label{hypdep}c_3+q_3>0.
	\end{equation}
	
	Then the process $(X,I)$ admits a unique invariant probability measure $\pi$ absolutely continuous with respect to Lebesgue's measure with support $\Gamma \times \{0,1\}$. The set $\Gamma $ will be defined \eqref{gamma}. 
	
	Moreover there exist strictly positive constants $C$ and $\gamma$ such that for any $t \geq0$ and for any $(x,i) \in M\times \{0,1\}$,
	\begin{equation}
	\label{cvexp} \p{ \bbP \Big((X_t,I_t) \in . \;\vert \, (X_0,I_0)=(x,i)\Big) - \pi }_{TV}\; \leq\; Ce^{-\gamma t}.\end{equation}
\end{theorem}


To prove this result, we establish intermediate results following the main steps than in Benaïm et al \cite{benaim2015qualitative}. We first construct a positively invariant compact set with respect to the flows associated with the dynamics of the process $X$ (Lemma \ref{compact}). Then we describe the set of accessible points of the process $(X_t,I_t)$ (Lemma \ref{thp}).
These two lemmas are the key points to prove Theorem \ref{th-mes-inv}.

\me Firstly, note that the assumption \eqref{hypdep} in Theorem \ref{th-mes-inv} ensures that the dynamics of $X_1$ and $X_2$ depend on the random dynamics of $I$. Indeed, in the opposite, the switch and division rates of healthy HSCs would be independent of the number of mutants red blood cells (cf \eqref{deftx}).

\bigskip \noindent Let us now specify the notations and state the two lemmas.

\me For $i\in \{0,1\}$, $\phi^i$ is the flow associated with the vector field $g(. ,i)$ defined by \eqref{g}. In other words, the function $t \mapsto \phi^i_t(x_0)=\phi^i(x_0,t)$ is the unique solution of the equation $$\frac{dx}{dt}(t) = g(x(t),i)=(g_1(x(t),i),\dots, g_3(x(t),i))^T$$ with $x(t=0)=x_0$.


\me Now we can construct a compact set $B$, positively invariant by the flows $(\phi^i)_{i\in \{0,1\}}$.

\begin{lemme}
	\label{compact}
	Let $B \subset E$ be the compact set $$B = [b_1,B_1] \times [b_2,\frac{c_1}{d}] \times [0,\frac{\cm{1}}{d_M}]$$ with $$ \displaystyle b_1 = \frac{a}{a+q_1 + q_2\, \frac{c_1}{d} + q_3 \,\frac{\cm{1}}{d_M} }
	\quad \text{,} \quad b_2 = \frac{c_1 \, b_1}{d \, (c_2 \,\frac{c_1}{d} + c_3 \,\frac{\cm{1}}{d_M} )}
	$$ and $$B_1=\frac{a}{a+q_1},$$ where the parameters are defined in \eqref{deftx} and satisfy the assumptions $(H1)$ and $(H2)$. 
	
	\medskip Then $B$ is positively invariant by the flows $\phi^i$, $i\in \{0,1\}$, i.e,
	$$\forall i\in \{0,1\},\; \forall t \geq 0, \quad \phi^i_t(B) \subset B. $$
	Moreover for any initial condition in $\tot{E}$, there exists $t \geq 0$ such that $$(X(t),I(t)) \in \tot{B}.$$ 
\end{lemme}

\begin{proof}
	Let us first note that the first two components of the vectors field $g(.,i)$ are independent of $i$. Thus 
	$$a.s. \quad \forall t \geq 0, \quad X_1(t) \leq 1 $$ and $$a.s. \quad \forall t \geq 0, \quad X_2(t) \leq \frac{c_1}{d}(1-e^{-dt}) + X_2(0) e^{-dt}.$$ We deduce that $$\forall (x_0,i_0) \in \tot{E}, \quad \exists t_1 \geq 0\; \text{ such that }\; \forall t \geq t_1, \quad X_2(t)  \leq \frac{c_1}{d}.$$
	
	\me Similarly, we know that,
	$$a.s. \quad \forall t \geq 0, \quad I(t) \leq 1 \;\; \text{ and } \;\; X_i(t)\geq0 \quad \text{ for } i \in \{2,3\}$$
	We deduce that almost surely $$\forall t \geq 0, \quad X_3(t) \leq \frac{\cm{1}}{d_M}(1-e^{-d_Mt}) + X_2(0) e^{-d_Mt} \;\; \text{ and } \;\; X_1(t)\leq \frac{a}{a+q_1}(1-e^{-(a+q_1)t}) + X_1(0) e^{-(a+q_1)t}$$ and then $$\forall (x_0,i_0) \in \tot{E}, \quad \exists t_2 \geq 0\; \text{ such that }\; \forall t \geq t_2, \quad  X_3(t)  \leq \frac{\cm{1}}{d_M}  \;\; \text{ and } \;\; X_1(t)\leq M_1. $$ 
	
	\me By similar arguments, we obtain
	$$\forall (x_0,i_0) \in \tot{E}, \quad \exists t_3 \geq 0\; \text{ such that }\; \forall t \geq t_3, \quad X_2(t)  \geq m_2 \quad \text{ and } \quad X_1(t)  \geq m_1. $$ 
	Indeed the functions $$ \displaystyle t \to m_1 \,(1-e^{\displaystyle-\big( a+q_1+q_2\, \frac{c_1}{d} + q_3 \,\frac{\cm{1}}{d_M} \big) \, t}) + x_1(0) e^{\displaystyle-\big( a+q_1+q_2\, \frac{c_1}{d} + q_3 \,\frac{\cm{1}}{d_M} \big) \,t} $$ and $$ t \to m_2 \, (1-e^{\displaystyle-dt}) + x_2(0) e^{\displaystyle-dt}$$ are respectively solutions of the following equations
	$$\frac{dy(t)}{dt}= a - (a+q_1 + q_2\frac{c_1}{d}+q_3\frac{c_{1,M}}{d_M})y(t) \quad \text{ and } \quad \displaystyle\frac{dz(t)}{dt}= \displaystyle\frac{c_1 m_1}{1+\displaystyle\frac{c_2c_1}{d} + \displaystyle\frac{c_3 \,c_{1,M}}{d_M}} - dz(t).
	$$
	
	Thus, the existence of a positively invariant compact set with respect to the flow $\phi^i$ has been proven for $i\in \{0,1\}$.
\end{proof}

\bigskip\noindent Without loss of generality, we assume in the following that $$(X(0),I(0))=(x_0,i_0) \in \tot{B}.$$

\me Let us define, as in \cite{benaim2015qualitative}, the notion of accessible point for the process $(X,I)$.

\me For all $n \in \bbN^*$, we define $$\bbT_n = \{ (\bar{i},\bar{u})=(i_1,\dots,i_{n}),(u_1,\dots,u_{n}) \in \{0,1\}^n \times \Rp{n} \}.$$ Then the trajectories of $(X,I)$ can be written using the flows $(\phi^i)_i$ as follows, \\ for $x\in B$ and $(\bar{i},\bar{u}) \in \bbT_n$, $$\Phi^{\bar{i}}_{\bar{u}}(x) = \phi^{i_{n}}_{u_{n}}\circ \dots \circ \phi^{i_1}_{u_1}(x).$$
For any $x \in B$, we define the positive orbits of $x$ by $$\gamma^+(x) = \{\Phi^{\bar{i}}_{\bar{u}}(x) : (\bar{i},\bar{u}) \in \displaystyle \bigcup_{n \in \bbN^*} \bbT_n \}.$$

\me A point $x$ is accessible from a singleton $\{y\}$ if $x\in \overline{ \gamma^+(y)}$. In a more general way, the set of accessible points is defined by \begin{equation}\label{gamma}
	\Gamma = \bigcap_{x \in B} \overline{ \gamma^+(x)}.
	\end{equation}

\me The following lemma allows us to describe more precisely $\Gamma$.
\begin{lemme}
	\label{thp}
	\begin{enumerate}
		\item The set of accessible points for the process $(X,I)$ is given by $\Gamma = \overline{ \gamma^+(p)}$ with $p=(p_1,p_2,0)\in B$ such that
	\begin{itemize}
		\item[$\bullet$] If $q_2 =0$, then
		
		$$ p_1=\frac{a}{a+q_1} \quad \text{ and } \quad  p_2 =\left\lbrace   \begin{array}{cl}
		\frac{c_1}{d}\, p_1 &\text{ if } c_2=0\\ 
		\\\displaystyle\frac{-d + \sqrt{d(d+4c_1c_2\,p_1)}}{2d\, c_2}  &\text{ else } \\
		\end{array}
		\right. .$$
		\item[$\bullet$] If $q_2 \neq 0$ and $c_2=0$, then
		
		$$\left\lbrace   \begin{array}{l}
		p_1= \displaystyle\frac{d}{2\,q_2 c_1}\, \Big(  \sqrt{(a+q_1)^2 +\frac{4aq_2c_1}{d}} - (a+q_1) \Big)\\ 
		\\p_2 =\displaystyle\frac{c_1}{d}\, p_1 \\
		\end{array}
		\right. .$$
		
		\item[$\bullet$] If $q_2 \neq 0$ and $c_2 \neq 0$ then $p_1$ is the unique solution of the following equation
		$$p_1=\frac{2dc_2a}{2dc_2(a + q_1 ) +q_2\, (\sqrt{d(d+4c_1c_2\,p_1)}-d)}$$ and $$p_2 = \frac{-d + \sqrt{d(d+4c_1c_2\,p_1)}}{2d\, c_2} .$$
	\end{itemize}
	
	\item The support of any invariant measure of the process $(X,I)$ is included in $\Gamma$. Moreover there exists an invariant probability measure with support equal to $\Gamma$.
	
	\end{enumerate}
\end{lemme}

\begin{proof}
	\begin{enumerate}
		\item 
	By definition of $\Gamma$, we know that for any $p \in B$, $\Gamma \subset \overline{ \gamma^+(p)}.$
	
	\me We will show that $\overline{ \gamma^+(p)}\subset\Gamma$ for $p$ unique solution of the equation $g(p,0)=0.$ 
	Let us start by showing the uniqueness of such an equilibrium. We will only detail here the case where $c_2 \neq 0$ and $q_2\neq 0$. The other cases can be proved by similar arguments.
	
	\me From the strict positivity of the constant $d$ and the expression of the function $g_3(p,0)$ we deduce that $p_3=0$. Then the couple $(p_1,p_2)$ is solution of the system $$\left\lbrace   \begin{array}{l}
	0 = a(1-p_1) - (q_1 + q_2 \, p_2 )\, p_1\\ 
	\\ 0 = \displaystyle\frac{c_1 \, p_1}{1+ c_2\, p_2} -d \, p_2 \\
	\end{array}
	\right. .$$
	Thus the real $p_2$ is a positive root of the polynomial $P(x) = dc_2 \, x^2 + d\, x - c_1 \, p_1$ whose discriminant $\Delta= d^2 +4c_1 d c_2 \,p_1$ is strictly positive. Therefore $p_2$ is uniquely defined, according to $p_1$, as the only positive root of $P$,
	$$p_2=p_2(p_1) = \frac{-d + \sqrt{d(d+4c_1c_2\,p_1)}}{2d\, c_2}.$$ Moreover, $p_1$ is the unique positive solution of the equation $$p_1 = \frac{a}{a+q_1 +q_2 \, p_2(p_1)}. $$
	Indeed, the function $p_1\in[0,1] \mapsto \displaystyle\frac{2dc_2a}{2dc_2(a + q_1) +q_2 \, (\sqrt{d(d+4c_1c_2\,p_1)}-d)}\in ]0,\frac{a}{a+q_1}]$ is strictly decreasing. So it intersects the identity function, which is strictly increasing, in a single point between $0$ and $1$. 
	
	\me To conclude it is enough to show that for every $\cU$ neighborhood of p, $$\forall x \in B, \; \exists u \in \Rp{} \quad \text{such that} \; \phi^{i=0}_u(x)\in \cU.$$
	The decrease towards $0$ of the third component of the flow $t \mapsto \phi^0_t(x)$ for any $x \in B$ is clear since the dynamics of this component is given by $\forall t \geq0, \quad \frac{dx_3(t)}{dt}=-d_M \,x_3(t)$.
	
	\medskip Thus we just have to study the behavior of the flow associated with the vectors field of $\Rp{2}$: $G=(g_1(.,.,.,0,0),g_2(.,.,.,0,0) )$.
	We have shown in Lemma \ref{compact} that this flow is contained in the compact set $$\widehat{B}=[b_1,B_1]\times [b_2,\frac{c_1}{d}]. $$ Then according to Poincaré-Bendixson's Theorem \cite{bendixson1901courbes}, either the vector field $G$ admits a periodic orbit, or for any initial condition belonging to $\widehat{B}$, the flow associated with $G$ converges to the unique stationary point belonging to $\widehat{B}$, i.e. $(p_1,p_2)$.
	
	The divergence of the vector field $G$ is not zero out of the compact set $\widehat{B}$ :
	$$div\,G(x) = -a-q_1\, x_2 - \frac{c_1c_2\, x_1}{(1+c_2 \, x_2)^2}-d <0.$$
	Therefore $G$ does not admit a periodic orbit (see Proposition \ref{Hulin} in Annex). 
	
	\me Hence we showed that $p \in \Gamma$ and then we finally conclude that
	$$ \overline{ \gamma^+(p)}=\Gamma.$$ 
	
	\item We have proved that $\Gamma \neq \emptyset$. Hence that ends the proof using \cite{benaim2015qualitative} Proposition 3.17 (i).
\end{enumerate}

\end{proof}

The set of accessible points $\Gamma$ gives us information on the support of any invariant probability measure of the PDMP $(X,I)$. Now we will prove, using \cite{benaim2018user} Corollary 2.7 and the previous lemmas, the existence and uniqueness of such a measure. Let us now prove Theorem \ref{th-mes-inv}.

\begin{proof}[Proof of Theorem \ref{th-mes-inv}]
	\me Following \cite{benaim2018user} Corollary 2.7, we need to check two conditions.

		\me The first condition is easy to satisfy.It consists in showing the existence of $s \in \bbR$ and $p \in \Gamma$ such that $s \,g(p,0) + (1-s) \, g(p,1) =0$. By Lemma \ref{thp}, we easily show that this condition is satisfied by the unique stationary point $p$ of $g(.,0)$ and $s=1$.
	
		\me Let us introduce a sequence of sets of vector fields, $\cG_0=\{g(.,i)\}_{i\in \{0,1\}}$ and for $j \geq1$, $\cG_{j+1} = \cG_j \cup \{ \,[g(.,i),V] \; ; \;\; V \in \cG_j, \; i\in \{0,1\} \; \}$ where $$[V,W](x) = DW(x) V(x) -DV(x)W(x), \quad \text{for } x \in \bbR^{3}.$$ The second condition (called weak bracket condition) is satisfied if \begin{equation}
		\label{braket}\exists x\in \Gamma \text{ such that } \quad Vect \{\, V(x) \; ; \quad V \in \displaystyle  \bigcup_{j \geq 0} \cG_j \, \} = \bbR^3.
		\end{equation}
		
\bigskip Once \eqref{braket} is checked, we apply \cite{benaim2018user} Corollary 2.7 and show the existence and uniqueness of an invariant probability measure $\pi$ for $(X,I)$. Moreover $\pi$ is absolutely continuous with respect to Lebesgue's measure and the exponential convergence \eqref{cvexp} holds.
	
	\me Then we deduce from Lemma \ref{thp}, that the support of this measure is given by $\tot{\Gamma}$.

	\me Let us prove \eqref{braket}.
	
	\me To simplify notation we introduce the following two functions $$ T(x)= 1+c_2\, x_2 + c_3\, x_3>0, \quad T_M(x)=1+c_{2,M}x_2+c_{3,M}x_3>0.$$
	Then the vector field $g(.,i)$ can be re-written for any $(x,i)\in \cE$, as follows, \begin{equation*}
	\label{g2}
	g(x,i) =\big( a(1-x_1) -q(x_2,x_3) x_1,\; \frac{c_1x_1}{T(x)} - d x_2,\; \frac{c_{1,M}}{T_M(x)} i - d_M x_3 \big)^T.
	\end{equation*}
	Its Jacobian matrix is given by
	$$Dg(.,i)(x)=\left(\begin{array}{ccc} -a-q(x_2,x_3) & -q_2\,x_1 & -q_3\, x_1\\ 
	\\ \displaystyle\frac{c_1}{T(x)} & -c_2\displaystyle \frac{ c_1 \, x_1}{T(x)^2} -d & -c_3\displaystyle\frac{ c_1\,x_1}{T(x)^2}\\
	\\ 0 & \displaystyle\frac{-c_{1,M} c_{2,M} \,i}{T_M(x)^2} & \displaystyle\frac{-c_{1,M} c_{3,M} \,i}{T_M(x)^2}  -d_M\end{array} \right).$$
	We know that
	$$ \forall x \in M, \quad [g(.,0),g(.,1)](x) = Dg(.,0)(x) \,g(x,1) -Dg(.,1)(x) \,g(x,0).$$
	Hence, computation gives
	$$[g(.,0),g(.,1)](x) = \Big( -q_3x_1\frac{c_{1,M}}{T_M(x)}, -c_3\frac{c_1x_1c_{1,M}}{T(x)^2\,T_M(x)},V_3(x)\Big)
	$$
	with $V_3(x)=-d_M \displaystyle\frac{c_{1,M}}{T_M(x)} + \frac{c_{1,M} c_{2,M} \,i}{T_M(x)^2} \Big(\frac{c_1x_1}{T(x)} - d x_2\Big) - \frac{c_{1,M} c_{3,M} \,i}{T_M(x)^2} d_M x_3 $.
	
	\me To check \eqref{braket}, we have to prove the existence of a point $x \in \Gamma$ such that
	$$\forall (\alpha_1,\alpha_2)\in \bbR^2, \quad [g(.,0),g(.,1)](x) \neq \alpha_1g(x,0)+\alpha_2g(x,1).
	$$
	To this end, we assume by contradiction that such point $x$ does not exist. 
	
	\me In other words we assume that for all $x\in \Gamma$, there exists $\alpha \in \bbR$ such that \begin{equation}\label{absurde1}
	\left\lbrace   \begin{array}{l}
	-q_3 x_1 \displaystyle \frac{ c_{1,M}}{T_M(x)} = \alpha \Big(a(1-x_1) - q(x_2,x_3)\, x_1\Big)\\ 
	\\- c_3\displaystyle \frac{c_1x_1c_{1,M}}{T(x)^2\,T_M(x)} = \alpha \Big(\frac{c_1\, x_1}{T(x)}-d\, x_2\Big) \\
	\end{array}
	\right. .\end{equation}
	Then for all $x \in \Gamma$, such that
	\begin{equation}
	\label{cond}
	a(1-x_1) \neq q(x_2,x_3)\, x_1,
	\end{equation} we obtain
	\begin{equation*}
	\frac{c_1c_3}{T(x)^2} = \frac{q_3}{a(1-x_1) - q(x_2,x_3)\, x_1} \Big(\frac{c_1\, x_1}{T(x)}-d\, x_2\Big).
	\end{equation*}
	According to \eqref{hypdep}, we deduce that for every $x \in \Gamma$ satisfying \eqref{cond},   \begin{equation}\label{absurde}x_1 = \frac{c_1c_3a + dq_3T(x)^2 x_2}{q_3 c_1 T(x)+c_1c_3(a+q(x_2,x_3))}.\end{equation}
	
	\me There are two different cases :
	\begin{itemize}
		\item[$\bullet$] $q_2\neq0$ or $q_3\neq0$. Let us introduce two points, $x=\phi^{i=1}_t(p)$ with $t>0$ and $y=\phi^{i=0}_u(x)$ with $u>0$, belonging to $\Gamma$. For $t$ large enough and $u$ small enough, the points $x$ and $y$ satisfy \eqref{cond}. Hence the points $x$ and $y$ also satisfy \eqref{absurde}. Otherwise we would get a contradiction with \eqref{absurde1}.
		
		\me Using that the vector field $g(x,i)$ depends on $i$ only through its third component, by definition of $y=\phi^{i=0}_u(x)$, we can easily find $u$ small enough such that $x_1=y_1$, $x_2=y_2$ and $x_3\neq y_3$. Hence we will show the existence of $u$ small enough for which $y=\phi^{i=0}_u(x)$ and $x$ do not both check \eqref{absurde}. Then we will have a contradiction with \eqref{absurde1} and \eqref{braket} will be shown.
		
		\me In order to study the variation of the function \begin{equation}\label{fct}x_3 \mapsto \frac{c_1c_3a + dq_3T(x)^2 x_2}{q_3 c_1 T(x)+c_1c_3(a+q(x_2,x_3))}\end{equation} let us re-write it as follows $$
		z \mapsto \frac{S_1 + S_2(1+S_3\, z)^2}{S_4+S_5\, z}
		$$ where $S_i\geq0$ and $S_4=q_3c_1(1+c_2x_2)+c_1c_3(a+q_1+q_2x_2) >0$.
		
		\me The derivative of such a function admits as numerator the following polynomial of degree up-bounded by $2$
		$$P(z)=2S_2S_3(1+S_3z)(S_4+S_5z) - (S_1+S_2(1+S_3z)^2)S_5.$$
		Such a polynomial cannot be equal to zero on the whole interval $[x_3,y_3]$, without admitting an infinity number of roots and hence being the zero polynomial on $\Rp{}$. Hence if $c_3 \neq0$, then the function \eqref{fct} cannot be constant. Thus if $c_3 \neq0$, there exists $u$ small enough such that $y \in \Gamma$ cannot check both conditions \eqref{cond} and \eqref{absurde}.
		
		\me Finally, if $c_3 =0$, by \eqref{hypdep}, $q_3\neq0$. In this case, there exists $t>0$ such that
		$$g_1(\phi^{i=1}_t(p),0)\neq0 \quad \text{and} \quad g_2(\phi^{i=1}_t(p),1)\neq 0.$$ Hence the condition \eqref{absurde1} cannot be satisfied in $x=\phi^{i=1}_t(p)$ since $x_1>0$ and $c_{1,M}>0$.
		
		\item[$\bullet$] $q_2=q_3=0$. In this case, by Lemma \ref{compact}, we know that there does not exist $x \in \Gamma$, satisfying \eqref{cond}. Indeed, for any $x \in \Gamma$, $x_1=\frac{a}{a+q_1}$. 
		
		\me In this case, the first component of $X$ is constant. Hence we can see the process $X$ as a two dimensional process. Then to check \eqref{braket}, we only need to prove the existence of $x\in \Gamma$ such that $$ \forall \alpha \in \bbR, \quad (g_2(x,0),g_3(x,0)) \neq \alpha\,(g_2(x,1),g_3(x,1)).$$ 
		\me We know that there exists $t>0$ such that $x=\phi^{i=1}_t(p)\in \Gamma$ satisfies this condition, for $p$ defined in Lemma \ref{thp}. Indeed, for any $z \in \Gamma$, $g_2(z,0)=g_2(z,1)$ and $g_3(z,0) \neq g_3(z,1)$. Hence, for $t$ large enough, we obtain that $g_2(x,0) =g_2(x,1) \neq 0$ and that $g_3(x,0) \neq g_3(x,1)$.
	\end{itemize}
	
\end{proof}

\section{Identification of the invariant probability measure $\pi$}
\label{chap3:sect:pers}

 We will finally give a description of the invariant probability measure using a stationary system of partial differential equations.
 
\begin{theorem}
	\label{ident}
	 We suppose that assumptions of Theorem \ref{th-mes-inv} are satisfied. Then there exists a unique couple of positive and integrable functions $(h_0,h_1)$, with support included in $B$, weak solution of the following stationary system of partial differential equations, for any $x \in int(B)= ]b_1, B_1[ \times]b_2, \frac{c_1}{d}[ \times ]0, \frac{c_{1,M}}{d_M}[$,
	\begin{equation}
	\label{h}
	\left\lbrace \begin{array}{lll}
	g(x,0)^T \nabla h_0(x) + h_0(x) \;\displaystyle \sum_{j=1}^3\partial_jg_j(x,0) &= q_M(x_2,x_3)\,h_{1}(x) - a_M h_0(x)\\ \\
	g(x,1)^T \nabla h_1(x) + h_1(x) \;\displaystyle \sum_{j=1}^3\partial_jg_j(x,1)  &= a_M\,h_0(x) - q_M(x_2,x_3) h_1(x)
	\end{array}\right.
	\end{equation} such that for any $i \in \{0,1\}$,  $$\pi(dx,\{i\}) = \displaystyle \sum_{j=0}^1 \delta_{ij} h_j(x) dx \quad \text{ and }\quad  \int_B (h_0(x)+h_1(x)) dx = 1.$$
	Here, $\delta_{ij}$ represents the Dirac measure, $\delta_{ij} = \left\lbrace \begin{array}{ll} 1 & \text{if } i=j \\ 0 & \text{else}\end{array}\right. .$

\end{theorem}

We will prove the existence first and then the uniqueness of a weak solution of the system \eqref{h} using Theorem \ref{th-mes-inv}.

\begin{proof}[Proof of Theorem \ref{ident}]
	\me Let $C^1(B)$ be the set of functions $f:\Rp{3} \times \{0,1\}$, such that for all $i\in \{0,1\}$,  $f(.,i)$ is $C^1$ on $B$ and $C^1_c(int(B))$ the set of functions $f:\Rp{3} \times \{0,1\}$, such that $f \in C^1(int(B))$ with a support included in $int(B)$.
	
	\me We are looking for $\pi \in \cP(\tot{B})$ such that \begin{equation}
	\label{anulgene}
	\forall f \in C^1(B), \quad \int_{\tot{B}} \cL f(z) \pi(dz) = 0.
	\end{equation} 
	We know, thanks to Theorem \ref{th-mes-inv} and using Radon-Nikodym Theorem, that there exist two positive and integrable functions $h_0$ and $h_1$, with support included in $B$, such that for $i \in \{0,1\}$,  $$\pi(dx,\{i\}) = \displaystyle \sum_{j=0}^1 \delta_{ij} h_j(x) dx.$$ Hence \begin{align}
	\label{gene}
	\int_{\tot{B}} \cL f(z) \pi(dz) = \int_{B} \sum_{i=0}^1  h_i(x) \, \big[ g(x,i)^T \nabla_x f(x,i) +Lf(x,i)\big] dx
	\end{align} where $$Lf(x,i)=a_M(1-i) \big( f(x,i+1)-f(x,i) \big) + q_M(x_2,x_3)\,i\, \big( f(x,i-1)-f(x,i) \big).
		$$
	Furthermore, $\forall x \in B$, \begin{align*}
	\sum_{i=0}^1 h_i(x) \,Lf(x,i) = \big( f(x,1)-f(x,0) \big)  \big( a_M \,h_0(x)- q_M(x_2,x_3) \, h_1(x) \big).
	\end{align*}
	Moreover, integrating by part \eqref{gene}, we obtain for any $i \in \{0,1\}$, $j\in\{1,2,3\}$ and for any function $f \in C^1_c(int(B))$,
	\begin{align*}
	\int_{B} h_i(x) \, g_j(x,i) \partial_{j}f(x,i) dx = 0 - \int_{B} \partial_j(h_i(x) g_j(x,i) ) f(x,i) dx.
	\end{align*}
	Hence, for any function $f \in C^1_c(int(B))$, we obtain
	\begin{align}
	\label{nul}
	\sum_{i=0}^1 \int_{B}  \big[ \displaystyle \sum_{j=1}^3\partial_j(h_i(x) g_j(x,i)) -  (1-2i)\, \big( q_M(x_2,x_3)\,h_{1}(x)-a_M\,h_0(x) \big) \big] f(x,i) dx=0.
	\end{align}
	
	\me Then we deduce that the couple of functions $(h_0,h_1)$ is a weak solution of the system of partial differential equations \eqref{h}.

\me This concludes the first part of the proof: the existence.

\me Now we will prove the uniqueness. 

\medskip Let $(h_0,h_1)$ be positive and integrable functions, weak solution of \eqref{h}, with support included in $B$ and such that $\int_B (h_0(x)+h_1(x)) dx=1$. We will prove the uniqueness of such functions using the uniqueness of invariant probability measure of Theorem \ref{th-mes-inv}.

\me We denote by $\tilde{\mathcal{L}}$ the following extension of the infinitesimal generator $\mathcal{L}$ on $\mathbb{R}_+^2\times \mathbb{R} \times \{0,1\}$, for all $f \in C^1(\mathbb{R}_+^2\times \mathbb{R} \times \{0,1\})$, $$ \tilde{\mathcal{L}}f(x_1,x_2,x_3,i)=\mathcal{L}f(x_1,x_2,\max(x_3,0),i).$$

\noindent Notice that Lemma \ref{compact} is also true for the extended process.

\me Then, let $\tilde{B}$ be a set such that $B \varsubsetneq \tilde{B} \subset \mathbb{R}_+^2\times \mathbb{R}$. We define $(\tilde{h_0},\tilde{h_1})$ as the extension of $(h_0,h_1)$ on $\tilde{B}$, i.e. we assume that $(\tilde{h_0},\tilde{h_1})$ is equal to $(h_0,h_1)$ on $B$ and equal to zero outside of $B$.
 Then by the same integration by parts as previously, we show that $\tilde{\pi}(dx,\{i\}) = \displaystyle \sum_{j=0}^1 \delta_{ij} \tilde{h_j}(x) dx$ defines an invariant probability measure for $\tilde{\mathcal{L}}$. Using Theorem \ref{th-mes-inv} on the extended infinitesimal generator $\tilde{\mathcal{L}}$, we deduce from the uniqueness of $\tilde{\pi}$, the uniqueness of $(h_0,h_1)$.\\
 
 This concludes the proof.
 
\end{proof}

\me Hence, we have mathematically described the macroscopic dynamics of all the cell types involved in the system when a cancer HSC able to become randomly quiescent appears. 

It would be interesting to compare these dynamics with biological observations of the symptoms of Myeloproliferative Neoplams. Then we could, if necessary, integrate into the model the details provided by works \cite{bonnet2019large} and \cite{notrepapier}, respectively on the phenomenon of cellular amplification between HSCs and red blood cells and on the regulation of hematopoietic stem cells.

\section*{Acknowledgements}

This research was led with financial support from ITMO Cancer of AVIESAN (Alliance Nationale pour les Sciences de la Vie et de la Santé, National Alliance for Life Sciences \& Health) within the framework of the Cancer Plan.

\section{Annexe}

\begin{prop}[\cite{Hulin}, Proposition 8.25]
	\label{Hulin}
	Let $X : U \to \bbR^2$ be a $C^1$ vector field such that $$X(x,y) = (f(x,y),g(x,y)).$$ We assume that $U$ is a simply connected open subset of $\bbR^2$. 
	\me If the divergence of $X$ $$div X := \frac{\partial f}{\partial x} + \frac{\partial g}{\partial y}$$ does not cancel, then $X$ does not have a periodic non-stationary orbit.
\end{prop}
\begin{proof}
	Let us introduce $\gamma$ a non trivial periodic orbit of $X$. Since the open set $U$ is simply connected, $\gamma$ borders a domain $\Omega \subset U$. Hence, we deduce from Green-Riemann formula that
	$$\int_{\gamma}f dy-gdx=\pm \int_{\Omega} (\frac{\partial f}{\partial x} + \frac{\partial g}{\partial y}) \, dxdy= \pm \int_{\Omega} div X \, dxdy
	$$ doesn't cancel (by assumption). By a variable change, we obtain a contradiction with the existence of $\gamma$, $$\int_{\gamma} f dy-gdx = \int_{\gamma}[ f\circ g(\gamma(t))-g\circ f(\gamma(t))] dt=0.$$

\end{proof}

	\bibliographystyle{plain}
	\bibliography{biblioPDMP}

\begin{thebibliography}{10}

\bibitem{ANDERSEN2020}
M.~Andersen, H.~C. Hasselbalch, L.~Kjær, V.~Skov, and J.~T. Ottesen.
\newblock Global dynamics of healthy and cancer cells competing in the
  hematopoietic system.
\newblock {\em Mathematical Biosciences}, 326:108372, 2020.

\bibitem{Bangsgaard2020}
K.~O. Bangsgaard, M.~Andersen, V.~Skov, L.~Kjær, H.~C. Hasselbalch, and J.~T.
  Ottesen.
\newblock Dynamics of competing heterogeneous clones in blood cancers explains
  multiple observations - a mathematical modeling approach.
\newblock {\em Mathematical Biosciences and Engineering}, 17(6):7645--7670,
  2020.

\bibitem{CoursMeleardVincent}
V.~Bansaye and S.~Méléard.
\newblock {\em Stochastic Models for Structured Populations: Scaling Limits and
  Long Time Behavior}.
\newblock Mathematical Biosciences Institute Lecture Series. Springer
  International Publishing, 2015.

\bibitem{benaim2018user}
M.~Bena{\"\i}m, T.~Hurth, and E.~Strickler.
\newblock A user-friendly condition for exponential ergodicity in randomly
  switched environments.
\newblock {\em Electronic Communications in Probability}, 23, 2018.

\bibitem{benaim2015qualitative}
M.~Bena{\"\i}m, S.~Le~Borgne, F.~Malrieu, and P.-A. Zitt.
\newblock Qualitative properties of certain piecewise deterministic markov
  processes.
\newblock In {\em Annales de l'IHP Probabilit{\'e}s et statistiques},
  volume~51, pages 1040--1075, 2015.

\bibitem{benaim2016lotka}
M.~Bena{\"\i}m and C.~Lobry.
\newblock Lotka--volterra with randomly fluctuating environments or “how
  switching between beneficial environments can make survival harder.
\newblock {\em The Annals of Applied Probability}, 26(6):3754--3785, 2016.

\bibitem{bendixson1901courbes}
I.~Bendixson.
\newblock Sur les courbes d{\'e}finies par des {\'e}quations
  diff{\'e}rentielles.
\newblock {\em Acta Mathematica}, 24:1--88, 1901.

\bibitem{besse2018stability}
A.~Besse, G.D. Clapp, S.~Bernard, F.E. Nicolini, D.~Levy, and T.~Lepoutre.
\newblock Stability analysis of a model of interaction between the immune
  system and cancer cells in chronic myelogenous leukemia.
\newblock {\em Bulletin of mathematical biology}, 80(5):1084--1110, 2018.

\bibitem{billingsley2013convergence}
P.~Billingsley.
\newblock {\em Convergence of probability measures}.
\newblock John Wiley \& Sons, 2013.

\bibitem{notrepapier}
C.~Bonnet, P.~Gou, S.~Girel, V.~Bansaye, C.~Lacout, K.~Bailly, M.-H.
  Schlagetter, E.~Lauret, S.~Méléard, and S.~Giraudier.
\newblock Combined biological and modeling approach of hematopoiesis: From
  native to stressed erythropoiesis.
\newblock {\em Available at SSRN: https://ssrn.com/abstract=3777468}.

\bibitem{bonnet2019large}
C.~Bonnet and S.~Méléard.
\newblock Large fluctuations in multi-scale modeling for rest erythropoiesis.
\newblock {\em 1912.05341, arXiv, math.PR}.

\bibitem{catlin2005kinetics}
S.N. Catlin, P.~Guttorp, and J.L. Abkowitz.
\newblock The kinetics of clonal dominance in myeloproliferative disorders.
\newblock {\em Blood}, 106(8):2688--2692, 2005.

\bibitem{cloez2017probabilistic}
B.~Cloez, R.~Dessalles, A.~Genadot, F.~Malrieu, A.~Marguet, and R.~Yvinec.
\newblock Probabilistic and piecewise deterministic models in biology.
\newblock {\em ESAIM: Proceedings and Surveys}, 60:225--245, 2017.

\bibitem{crudu}
A.~Crudu, A.~Debussche, A.~Muller, and O.~Radulescu.
\newblock Convergence of stochastic gene networks to hybrid piecewise
  deterministic processes.
\newblock {\em The Annals of Applied Probability}, 22(5):1822--1859, 2012.

\bibitem{davis1984piecewise}
MHA Davis.
\newblock Piecewise-deterministic markov processes: A general class of
  non-diffusion stochastic models.
\newblock {\em Journal of the Royal Statistical Society: Series B
  (Methodological)}, 46(3):353--376, 1984.

\bibitem{davis1993markov}
MHA Davis.
\newblock Markov models and optimization. chapman \&amp; hall, 1993.

\bibitem{dingli2006successful}
D.~Dingli and F.~Michor.
\newblock Successful therapy must eradicate cancer stem cells.
\newblock {\em Stem cells}, 24(12):2603--2610, 2006.

\bibitem{dingli2011stochastic}
D.~Dingli and J.M. Pacheco.
\newblock Stochastic dynamics and the evolution of mutations in stem cells.
\newblock {\em BMC biology}, 9(1):41, 2011.

\bibitem{dingli2007symmetric}
D~Dingli, A~Traulsen, and F~Michor.
\newblock (a) symmetric stem cell replication and cancer.
\newblock {\em PLoS computational biology}, 3(3), 2007.

\bibitem{fasano2017blood}
A.~Fasano and A.~Sequeira.
\newblock Blood and cancer.
\newblock In {\em Hemomath}, pages 295--330. Springer, 2017.

\bibitem{fontbona2012quantitative}
J.~Fontbona, H.~Gu{\'e}rin, and F.~Malrieu.
\newblock Quantitative estimates for the long-time behavior of an ergodic
  variant of the telegraph process.
\newblock {\em Advances in Applied Probability}, 44(4):977--994, 2012.

\bibitem{glauche2012therapy}
I.~Glauche, K.~Horn, M.~Horn, L.~Thielecke, M.A.G. Essers, A.~Trumpp, and
  I.~Roeder.
\newblock Therapy of chronic myeloid leukaemia can benefit from the activation
  of stem cells: simulation studies of different treatment combinations.
\newblock {\em British journal of cancer}, 106(11):1742--1752, 2012.

\bibitem{haeno2009progenitor}
H.~Haeno, R.L. Levine, D.G. Gilliland, and F.~Michor.
\newblock A progenitor cell origin of myeloid malignancies.
\newblock {\em Proceedings of the National Academy of Sciences},
  106(39):16616--16621, 2009.

\bibitem{horn2008mathematical}
M.~Horn, M.~Loeffler, and I.~Roeder.
\newblock Mathematical modeling of genesis and treatment of chronic myeloid
  leukemia.
\newblock {\em Cells Tissues Organs}, 188(1-2):236--247, 2008.

\bibitem{Hulin}
D.~Hulin.
\newblock Equations différentielles ordinaires. etudes qualitatives.
\newblock {\em Cours L3 MFA Université Paris-Sud}.

\bibitem{jacobsen2006point}
M.~Jacobsen.
\newblock {\em Point process theory and applications: marked point and
  piecewise deterministic processes}.
\newblock Springer Science \& Business Media, 2006.

\bibitem{JoffeMetivier}
A.~Joffe and M.~Métivier.
\newblock Weak convergence of sequences of semimartingales with applications to
  multitype branching processes.
\newblock {\em Advances in Applied Probability}, 18(1):20–65, 1986.

\bibitem{komarova2007effect}
N.L. Komarova and D.~Wodarz.
\newblock Effect of cellular quiescence on the success of targeted cml therapy.
\newblock {\em PLoS one}, 2(10), 2007.

\bibitem{lepelmar}
J.-P. Lepeltier and B.~Marchal.
\newblock Probl{\`e}me des martingales et {\'e}quations diff{\'e}rentielles
  stochastiques associ{\'e}es {\`a} un op{\'e}rateur
  int{\'e}gro-diff{\'e}rentiel.
\newblock In {\em Annales de l'IHP Probabilit{\'e}s et statistiques},
  volume~12, pages 43--103, 1976.

\bibitem{li2017threshold}
D.~Li and S.~Liu.
\newblock Threshold dynamics and ergodicity of an sirs epidemic model with
  markovian switching.
\newblock {\em Journal of Differential Equations}, 263(12):8873--8915, 2017.

\bibitem{mackey2013dynamic}
M.C. Mackey, M.~Tyran-Kaminska, and R.~Yvinec.
\newblock Dynamic behavior of stochastic gene expression models in the presence
  of bursting.
\newblock {\em SIAM Journal on Applied Mathematics}, 73(5):1830--1852, 2013.

\bibitem{malrieu2015some}
F.~Malrieu.
\newblock Some simple but challenging markov processes.
\newblock In {\em Annales de la Facult{\'e} des sciences de Toulouse:
  Math{\'e}matiques}, volume~24, pages 857--883, 2015.

\bibitem{michor2008mathematical}
F.~Michor.
\newblock Mathematical models of cancer stem cells.
\newblock {\em Journal of Clinical Oncology}, 26(17):2854--2861, 2008.

\bibitem{pedersen2023}
Rasmus~K Pedersen, Morten Andersen, Thomas Stiehl, and Johnny~T Ottesen.
\newblock Understanding hematopoietic stem cell dynamics—insights from
  mathematical modelling.
\newblock {\em Current Stem Cell Reports}, pages 1--8, 2023.

\bibitem{roeder2006dynamic}
I.~Roeder, M.~Horn, I~Glauche, A~Hochhaus, M~C Mueller, and M~Loeffler.
\newblock Dynamic modeling of imatinib-treated chronic myeloid leukemia:
  functional insights and clinical implications.
\newblock {\em Nature medicine}, 12(10):1181--1184, 2006.

\bibitem{stiehl2012mathematical}
T.~Stiehl and A.~Marciniak-Czochra.
\newblock Mathematical modeling of leukemogenesis and cancer stem cell
  dynamics.
\newblock {\em Mathematical Modelling of Natural Phenomena}, 7(1):166--202,
  2012.

\bibitem{yvinec2014adiabatic}
R.~Yvinec, C.~Zhuge, J.~Lei, and M.C. Mackey.
\newblock Adiabatic reduction of a model of stochastic gene expression with
  jump markov process.
\newblock {\em Journal of mathematical biology}, 68(5):1051--1070, 2014.

\end{thebibliography}
\end{document}